\documentclass[10pt]{amsart}
\usepackage{amssymb}
\usepackage{epsfig}
\usepackage{latexsym}
\usepackage{amsmath,amssymb,amsfonts,amsthm,graphics}%,latexcad
\usepackage{eucal}%    caligraphic-euler fonts: \mathcal{ }
\usepackage{eufrak}%   frak-euler        fonts: \mathfrak{ }
\usepackage[all]{xypic}
\usepackage{xspace}
\usepackage{color}
\theoremstyle{plain}
\newtheorem{theorem}{Theorem}

\newtheorem{proposition}{Proposition}

\theoremstyle{definition}

\theoremstyle{remark}

\numberwithin{equation}{section}

\begin{document}
%%%%%%%%%%%%%%%%%%%%%%%%%%%%%%%%%%%%%%%%%%%%%%%%%%%%%%%%%%%%

%%%
%%%%%%%%%%%%%%%%%%%%%%%%%%%%%%%%%%%%%%%%%%%%%%%%%%%%%%%%%%%%%%%%%%%%%%%%%%
%%%%%

\title[Loops in canonical RNA pseudoknot structures]
      {Loops in canonical RNA pseudoknot structures}
\author{Markus E. Nebel$^{\dagger}$, Christian M. Reidys$^{*}$
        and Rita R. Wang $^{*}$}
\address{ $^{*}$ Center for Combinatorics, LPMC-TJKLC \\
          Nankai University  \\
          Tianjin 300071\\
          P.R.~China\\
          Phone: *86-22-2350-6800\\
          Fax:   *86-22-2350-9272 \\
          $^{\dagger}$ TU Kaiserslautern \\
          67663 Kaiserslautern \\
          Germany \\
          %Phone: *86-22-2350-6800 \\
          %Fax:   *86-22-2350-9272
}
\email{reidys@nankai.edu.cn}
\thanks{}
\date{October, 2009}
\begin{abstract}
In this paper we compute the limit distributions of the numbers of
hairpin-loops, interior-loops and bulges in $k$-noncrossing RNA structures.
The latter are coarse grained RNA structures allowing for cross-serial
interactions, subject to the constraint that there are at most
$k-1$ mutually crossing arcs in the diagram representation of the molecule.
We prove central limit theorems by means of studying the corresponding
bivariate generating functions. These generating functions are obtained
by symbolic inflation of ${\sf lv}_k^{\sf 5}$-shapes \cite{Reidys:shape}.
\end{abstract}
\maketitle {{\small
%\tableofcontents
}}

%%%
%%%%%%%%%%%%%%%%%%%%%%%%%%%%%%%%%%%%%%%%%%%%%%%%%%%%%%%%%%%%%%%%%%%%%%%%
%%%
%%%
\section{Introduction}\label{S:intro}
%%%
%%%%%%%%%%%%%%%%%%%%%%%%%%%%%%%%%%%%%%%%%%%%%%%%%%%%%%%%%%%%%%%%%%%%%%%%
%%%
%%%

An RNA molecule is a sequence of the four nucleotides {\bf A}, {\bf
G}, {\bf U}, {\bf C} together with the Watson-Crick ({\bf A-U}, {\bf
G-C}) and {\bf U-G} base pairing rules. The sequence of bases is
called the primary structure of the RNA molecule. Two bases in the
primary structure which are not adjacent may form hydrogen bonds
following the Watson-Crick base pairing rules. Three decades ago
Waterman {\it et al.} \cite{Kleitman:70,Nussinov,Waterman:79}
analyzed RNA secondary structures. Secondary structures are coarse
grained RNA contact structures. They can be represented as diagrams
and planar graphs, see Fig.~\ref{F:second}.
%%%
%%%%%%%%%%%%%%%%%%%%%%%%%%%%%%%%%%%%%%%%%%%%%%%%%%%%%%%%%%%%%%%%%%%%%%%
%%%
\begin{figure}[ht]
\centerline{\includegraphics[width=0.8\textwidth]{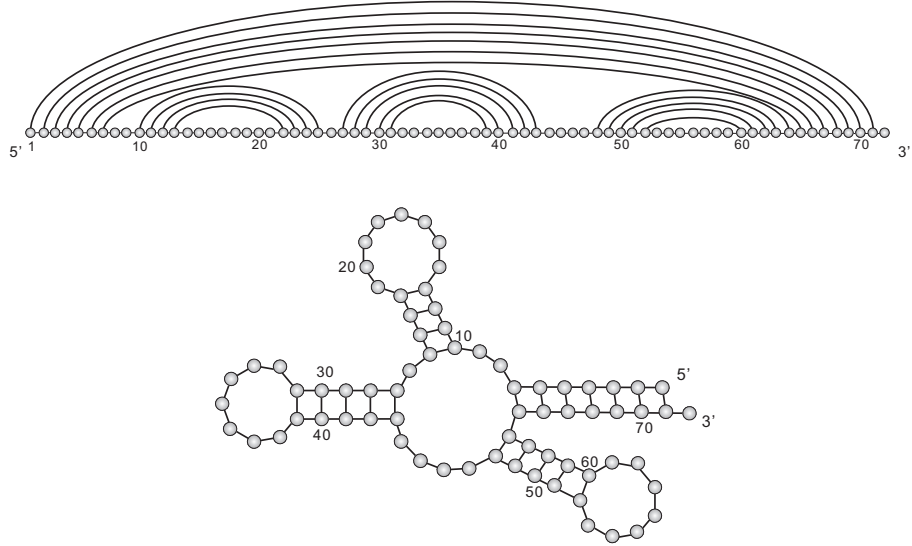}}
\caption{\small The Sprinzl tRNA RD7550 secondary structure
represented as $2$-noncrossing diagram (top) and planar graph
(bottom).} \label{F:second}
\end{figure}
%%%
%%%%%%%%%%%%%%%%%%%%%%%%%%%%%%%%%%%%%%%%%%%%%%%%%%%%%%%%%%%%%%%%%%%%%%
%%%
Diagrams are labeled graphs over the vertex set $[n]=\{1, \dots,
n\}$ with vertex degrees $\le 1$, represented by drawing its
vertices on a horizontal line and its arcs $(i,j)$ ($i<j$), in the
upper half-plane, see Fig.~\ref{F:second} and Fig.~\ref{F:canonical}.
Here, vertices and arcs correspond to the nucleotides {\bf A}, {\bf
G}, {\bf U}, {\bf C} and Watson-Crick ({\bf A-U}, {\bf G-C}) and
({\bf U-G}) base pairs, respectively. In a diagram two arcs
$(i_1,j_1)$ and $(i_2,j_2)$ are called crossing if $i_1<i_2<j_1<j_2$
holds. Accordingly, a $k$-crossing is a sequence of arcs
$(i_1,j_1),\dots,(i_k,j_k)$ such that
$i_1<i_2<\dots<i_k<j_1<j_2<\dots <j_k$, see Fig.~\ref{F:canonical}.
We call diagrams containing at most $(k-1)$-crossings,
$k$-noncrossing diagrams ($k$-noncrossing partial matchings). The
length of an arc $(i,j)$ is given by $j-i$, characterizing the
minimal length of a hairpin loop. A stack of length $\tau$ is a
sequence of ``parallel'' arcs of the form
\begin{equation}\label{E:stack}
((i,j),(i+1,j-1),\ldots, (i+(\tau-1),j-(\tau-1))),
\end{equation}
and we denote it by $S_{i,j}^{\tau}$. We call an arc of length one a
$1$-arc. A $k$-noncrossing, $\tau$-canonical RNA structure is a
$k$-noncrossing diagram without $1$-arcs, having a minimum
stack-size of $\tau$, see Fig.~\ref{F:canonical}.
Let $\mathcal{T}_{k,\tau}(n)$ denote the set of $k$-noncrossing,
$\tau$-canonical RNA structures of length $n$ and let
$\mathsf{T}_{k,\tau}(n)$ denote their number.
%%%
%%%%%%%%%%%%%%%%%%%%%%%%%%%%%%%%%%%%%%%%%%%%%%%%%%%%%%%%%%%%%%%%%%
%%%
\begin{figure}[ht]
\centerline{\includegraphics[width=0.65\textwidth]{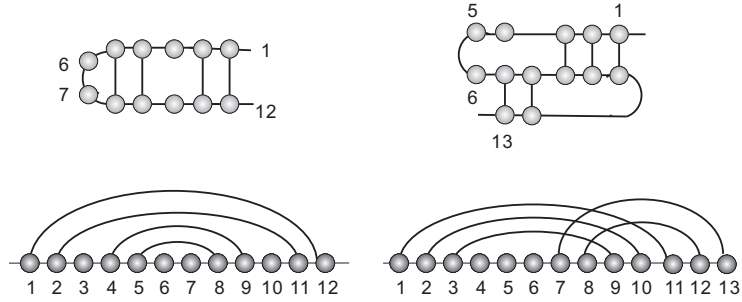}}
\caption{\small A $2$-noncrossing, $2$-canonical RNA structure
(left) and a $3$-noncrossing, $2$-canonical RNA structure (right)
represented as planer graphs (top) and diagrams (bottom).}
\label{F:canonical}
\end{figure}
%%%
%%%%%%%%%%%%%%%%%%%%%%%%%%%%%%%%%%%%%%%%%%%%%%%%%%%%%%%%%%%%%%%%%%%%%%

We next introduce the following structural elements of $k$-noncrossing,
$\tau$-canonical RNA structures, see Fig.~\ref{F:loop11} and
Fig.~\ref{F:loop}.
%%%
%%%%%%%%%%%%%%%%%%%%%%%%%%%%%%%%%%%%%%%%%%%%%%%%%%%%%%%%%%%%%%%%%%
%%%
\begin{figure}[ht]
\centerline{\includegraphics[width=0.8\textwidth]{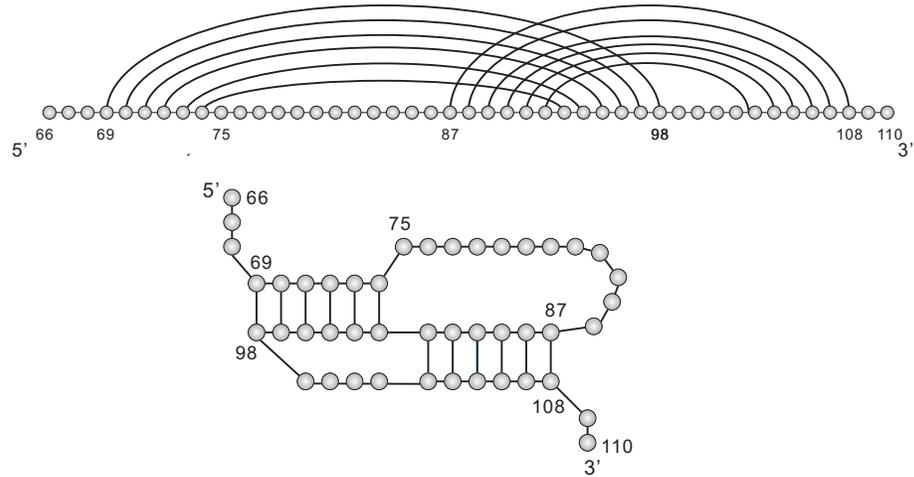}}
\caption{\small $3$-noncrossing, $6$-canonical structures:
the pseudoknot structure of the PrP-encoding mRNA represented as diagrams (top) and planer graphs (bottom)..}
\label{F:loop11}
\end{figure}
%%%
%%%%%%%%%%%%%%%%%%%%%%%%%%%%%%%%%%%%%%%%%%%%%%%%%%%%%%%%%%%%%%%%%%%%%%
%%%
Let $[i,j]$ denote an interval, i.e. a sequence of consecutive isolated
vertices $(i,i+1,\ldots,j-1,j)$. We consider, see Fig.~\ref{F:loop}
\begin{enumerate}
\item a {\it hairpin-loop} is  a pair
$$
((i,j),[i+1,j-1]).
$$
\item an {\it interior-loop} is a sequence
$$((i_1,j_1),[i_1+1,i_2-1],(i_2,j_2),[j_2+1,j_1-1]),$$
where $(i_2,j_2)$ is nested in $(i_1,j_1)$.
\item a {\it bulge} is a sequence
$$((i_1,j_1),[i_1+1,i_2-1],(i_2,j_1-1))
\quad\text{or}\quad ((i_1,j_1),(i_1+1,j_2),[j_2+1,j_1-1]).$$
\item a {\it stem} is a sequence of stacks
$$\left(S_{i_1,j_1}^{\tau_1},
S_{i_2,j_2}^{\tau_2},\ldots,S_{i_{s},j_{s}}^{\tau_{s}}\right)$$
where the stack $S_{i_{m},j_{m}}^{\tau_m}$ is nested
in $S_{i_{m-1},j_{m-1}}^{\tau_{m-1}}$, $2\leq m\leq s$
and there are no arcs of the form $(i_1-1, j_1+1)$ and
$(i_s+\tau_s, j_s-\tau_s)$.
\end{enumerate}
%%%
%%%%%%%%%%%%%%%%%%%%%%%%%%%%%%%%%%%%%%%%%%%%%%%%%%%%%%%%%%%%%%%%%%
%%%
\begin{figure}[ht]
\centerline{\includegraphics[width=0.55\textwidth]{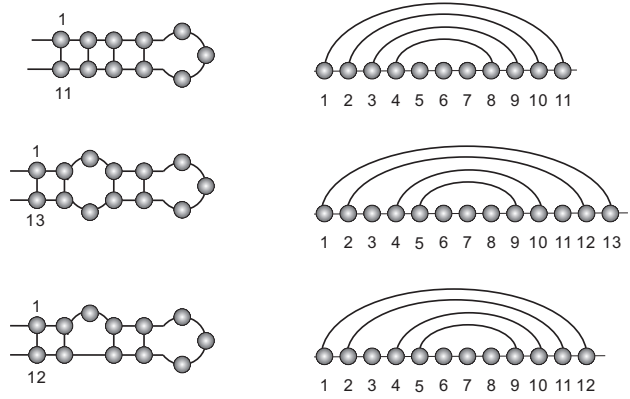}}
\caption{\small The loop-types: hairpin-loop (top), interior-loop
(middle) and bulge (bottom).} \label{F:loop}
\end{figure}
%%%
%%%%%%%%%%%%%%%%%%%%%%%%%%%%%%%%%%%%%%%%%%%%%%%%%%%%%%%%%%%%%%%%%%%%%%
%%%
%%%
%%%%%%%%%%%%%%%%%%%%%%%%%%%%%%%%%%%%%%%%%%%%%%%%%%%%%%%%%%%%%%%%%%%%%%%%
%%%
%%%
\section{Preliminaries}\label{S:preli}
%%%
%%%%%%%%%%%%%%%%%%%%%%%%%%%%%%%%%%%%%%%%%%%%%%%%%%%%%%%%%%%%%%%%%%%%%%%%
%%%
Let $f_k(n,\ell)$ denote the number of $k$-noncrossing diagrams on
$n$ vertices having exactly $\ell$ isolated vertices. A diagram
without isolated points is called a matching. The exponential
generating function of $k$-noncrossing matchings satisfies the
following identity \cite{Chen,Grabiner:93a,Reidys:07pseu}
\begin{equation}\label{E:ww0}
{{\bf H}_k(z)}=\sum_{n\ge 0} f_{k}(2n,0)\cdot\frac{z^{2n}}{(2n)!}  =
\det[I_{i-j}(2z)-I_{i+j}(2z)]|_{i,j=1}^{k-1}
\end{equation}
where $I_{r}(2z)=\sum_{j \ge 0}\frac{z^{2j+r}}{{j!(j+r)!}}$ is the
hyperbolic Bessel function of the first kind of order $r$.
Eq.~(\ref{E:ww0}) allows us to conclude that the ordinary generating
function
\begin{equation*}
{\bf F}_k(z)=\sum_{n\ge 0}f_k(2n,0) z^{n}
\end{equation*}
is $D$-finite \cite{Stanley:80}. This follows from the fact that
$I_{r}(2z)$ is $D$-finite and $D$-finite power series form an
algebra \cite{Stanley:80}. Consequently, there exists some $e\in
\mathbb{N}$ such that
\begin{equation}\label{E:ODE}
q_{0,k}(z)\frac{d^e}{d z^e}{\bf F}_k(z)+q_{1,k}(z)\frac{d^{e-1}}{d
z^{e-1}}{\bf F}_k(z)+\cdots+q_{e,k}(z){\bf F}_k(z)=0,
\end{equation}
where $q_{j,k}(z)$ are polynomials and $q_{0,k}(z)\neq 0$.
The ordinary differential equations (ODE) for ${\bf F}_k(z)$,
where $2\leq k\leq 7$ are obtained by the
MAPLE package {\tt GFUN} from the exact data of $f_k(2n,0)$. They are
verified by first deriving the corresponding $P$-recursions
\cite{Stanley:80} for $f_k(2n,0)$
second transforming these $P$-recursions into $P$-recursions of
$f_k(2n,0)/(2n)!$ and third deriving the corresponding ODEs for
${\bf H}_k(z)$ and verifying that the RHS of eq.~(\ref{E:ww0}) is a solution.
The key point is that any singularity of ${\bf F}_k(z)$ is contained in the
set of roots of $q_{0,k}(z)$ \cite{Stanley:80}, which we denote by
$R_k$. For $2\leq k\leq 7$, we give the polynomials $q_{0,k}(z)$ and
their roots in Table \ref{Table:polyroot}.
%%%%%%%%
%%%%%%%%%%%%%%%%%%%%%%%%%%%%%%%%%%%%%%%%%%%%%%%%%%%%%%%%%%%%%%%%%%%%%
%%%%%%%%
{\small
\begin{table}
\begin{center}
\begin{tabular}{cll}
\hline
$k$ & $q_{0,k}(z)$ & $R_k$  \\
\hline
$2$ & $(4z-1)z$ & $\{\frac{1}{4}\}$\\
$3$ & $(16z-1)z^2$ & $\{\frac{1}{16}\}$\\
$4$ & $(144z^2-40z+1)z^3$ & $\{\frac{1}{4},\frac{1}{36}\}$\\
$5$ & $(1024z^2-80z+1)z^4$ & $\{\frac{1}{16},\frac{1}{64}\}$\\
$6$ & $(14400z^3-4144z^2+140z-1)z^5$ & $\{\frac{1}{4},\frac{1}{36},\frac{1}{100}\}$\\
$7$ & $(147456z^3-12544z^2+224z-1)z^6$ & $\{\frac{1}{16},\frac{1}{64},\frac{1}{144}\}$\\
\hline
\end{tabular}
\centerline{}
\smallskip
\caption{\small We present the polynomials $q_{0,k}(z)$  and their
nonzero roots obtained by the MAPLE package {\tt GFUN}. }
\label{Table:polyroot}
\end{center}
\end{table}
}
In \cite{Reidys:08k} we showed that for arbitrary $k$
\begin{equation}\label{E:theorem}
f_{k}(2n,0) \, \sim  \, \widetilde{c}_k  \, n^{-((k-1)^2+(k-1)/2)}\,
(2(k-1))^{2n},\qquad \widetilde{c}_k>0
\end{equation}
in accordance with the fact that ${\bf F}_k(z)$ has the unique
dominant singularity $\rho_k^2$, where $\rho_k=1/(2k-2)$.

We next introduce a central limit theorem due to Bender \cite{Bender:73}.
It is proved by analyzing the characteristic function by the
L\'{e}vy-Cram\'{e}r Theorem (Theorem IX.4 in \cite{Flajolet:07a}).
%%%%%
%%%%%%%%%%%%%%%%%%%%%%%%%%%%%%%%%%%%%%%%%%%%%%%%%%%%%%%%%%%%%%%%
%%%%
\begin{theorem}\label{T:normal}
Suppose we are given the bivariate generating function
\begin{equation}
f(z,u)=\sum_{n,t\geq 0}f(n,t)\,z^n\,u^t,
\end{equation}
where $f(n,t)\geq 0$
and $f(n)=\sum_tf(n,t)$. Let $\mathbb{X}_n$ be a r.v.~such that
$\mathbb{P}(\mathbb{X}_n=t)=f(n,t)/f(n)$. Suppose
\begin{equation}\label{E:knackpunkt}
[z^n]f(z,e^s)\sim c(s) \, n^{\alpha}\,  \gamma(s)^{-n}
\end{equation}
uniformly in $s$ in a neighborhood of $0$, where $c(s)$ is continuous
and nonzero near $0$, $\alpha$ is a constant, and $\gamma(s)$ is analytic
near $0$.
Then there exists a pair $(\mu,\sigma)$ such that the normalized
random variable
\begin{equation}\label{E:normalized}
\mathbb{X}^*_{n}=\frac{\mathbb{X}_{n}- \mu \, n}{\sqrt{{n\,\sigma}^2 }}
\end{equation}
has asymptotically normal distribution with parameter $(0,1)$.
That is we have
\begin{equation}\label{E:converge-2}
\lim_{n\to\infty}
\mathbb{P}\left(
\mathbb{X}^*_{n} < x \right)  =
\frac{1}{\sqrt{2\pi}} \int_{-\infty}^{x}\,e^{-\frac{1}{2}c^2} dc \,
\end{equation}
where $\mu$ and $\sigma^2$ are given by
\begin{equation}\label{E:was}
\mu= -\frac{\gamma'(0)}{\gamma(0)}
\quad \text{\it and} \quad \sigma^2=
\left(\frac{\gamma'(0)}{\gamma(0)}
\right)^2-\frac{\gamma''(0)}{\gamma(0)}.
\end{equation}
\end{theorem}
%%%%%%%%%%%%%%%%%%%%%%%%%%%%%%%%%%%%%%%%%%%%%%%%%%%%%%%%%%%%%%%%%%%%%%%%
%%%
The crucial points for applying Theorem~\ref{T:normal} are
{\sf (a)} eq.~(\ref{E:knackpunkt})
\begin{equation*}
[z^n]f(z,e^s)\sim c(s) \, n^{\alpha}\,  \gamma(s)^{-n},
\end{equation*}
uniformly in $s$ in a neighborhood of $0$, where $c(s)$ is continuous
and nonzero near $0$ and $\alpha$ is a constant and {\sf (b)}
the analyticity of $\gamma(s)$ in $s$ near $0$.
In the following, we have generating functions of the form
${\bf F}_k(\psi(z,s))$. In this situation, Theorem~\ref{T:uniform} below
guarantees under specific conditions
\begin{equation*}
[z^n]{\bf F}_k(\psi(z,s)) \sim A(s)\,n^{-((k-1)^2+(k-1)/2)}
\left(\frac{1}{\gamma(s)}\right)^n,\quad \text{$A(s)$ continuous},
\end{equation*}
for $2\leq k\leq 7$.
The analyticity of $\gamma(s)$ is guaranteed by the analytic implicit
function theorem \cite{Flajolet:07a}.
%%%
%%%%%%%%%%%%%%%%%%%%%%%%%%%%%%%%%%%%%%%%%%%%%%%%%%%%%%%%%%%%%%%%%%%%%%%%%
%%%
\begin{theorem}\label{T:uniform}\cite{Reidys:lego2}
Suppose $2\leq k\leq 7$. Let $\psi(z,s)$ be an analytic function in a domain
\begin{equation}
\mathcal{D}=\{(z,s)||z|\leq r,|s|<\epsilon\}
\end{equation}
such that $\psi(0,s)=0$. In addition suppose $\gamma(s)$ is the
unique dominant singularity of ${\bf F}_k(\psi(z,s))$ and analytic
solution of $\psi(\gamma(s),s)=\rho_k^2$, $|\gamma(s)|\leq r$,
$\partial_z\psi(\gamma(s),s)\neq 0$ for $\vert s\vert<\epsilon$.
Then ${\bf F}_k(\psi(z,s))$ has a singular expansion and
\begin{equation}\label{E:LL}
[z^n]{\bf F}_k(\psi(z,s)) \sim A(s)\,n^{-((k-1)^2+(k-1)/2)}
\left(\frac{1}{\gamma(s)}\right)^n\quad \text{for some continuous
$A(s)\in\mathbb{C}$},
\end{equation}
uniformly in $s$ contained in a small neighborhood of $0$.
\end{theorem}
%%%%%
%%%%%%%%%%%%%%%%%%%%%%%%%%%%%%%%%%%%%%%%%%%%%%
%%%
To keep the paper selfcontained we give a direct proof of
Theorem~\ref{T:uniform} in Section~\ref{S:proof}. This avoids
calling upon generic results, such as the uniformity Lemma of
singularity analysis \cite{Flajolet:07a}.

%%%
%%%%%%%%%%%%%%%%%%%%%%%%%%%%%%%%%%%%%%%%%%%%%%%%%%%%%%%%%%%%%%%%%%%%%%%%
%%%
\section{The generating function}\label{S:comb}
%%%
%%%%%%%%%%%%%%%%%%%%%%%%%%%%%%%%%%%%%%%%%%%%%%%%%%%%%%%%%%%%%%%%%%%%%%%%
%%%

In this section we compute the bivariate generating functions
of hairpin-loops, interior-loops and bulges. Let $h_{k,\tau}(n,t)$,
$i_{k,\tau}(n,t)$ and $b_{k,\tau}(n,t)$ denote the numbers of
$k$-noncrossing, $\tau$-canonical RNA structures of length $n$
with $t$ hairpin-loops, interior-loops and bulges. We set
\begin{eqnarray}
{\bf H}_{k,\tau}(z,u_1)&=&\sum_{n\geq 0}\sum_{t\geq 0}
h_{k,\tau}(n,t)z^n\,u_1^t,\\
{\bf I}_{k,\tau}(z,u_2)&=&\sum_{n\geq 0}\sum_{t\geq 0}
i_{k,\tau}(n,t)z^n\,u_2^t,\\
{\bf B}_{k,\tau}(z,u_3)&=&\sum_{n\geq 0}\sum_{t\geq
0}b_{k,\tau}(n,t)z^n\,u_3^t.
\end{eqnarray}
In order to derive the above generating functions we use symbolic
enumeration \cite{Flajolet:07a}.
A combinatorial class is a set of finite size with the
definition of size function of its elements, whose elements are all
finite size and the number of certain size elements is finite.
Suppose $\mathcal{C}$ be a combinatorial class and $c\in
\mathcal{C}$. We denote the size of $c$ by $|c|$. There are two
special combinatorial classes $\mathcal{E}$ and $\mathcal{Z}$ which
respectively contains only an element of size $0$ and an element of
size $1$. The subset of $\mathcal{C}$ which contains all the
elements of size $n$ in $\mathcal{C}$ is denoted by $\mathcal{C}_n$.
Then the generating function of a combinatorial class $\mathcal{C}$ is
\begin{equation}
{\bf C}(z)=\sum_{c\in\mathcal{C}}z^{|c|}=\sum_{n\geq 0}C_n\, z^n,
\end{equation}
where $\mathcal{C}_n\subset \mathcal{C}$ and $C_n=|\mathcal{C}_n|$.
In particular the generating functions of $\mathcal{E}$ and $\mathcal{Z}$
are given by ${\bf E}(z)=1$ and ${\bf Z}(z)=z$. For any two combinatorial
classes $\mathcal{C}$,
$\mathcal{D}$, we have the following operations:
\begin{itemize}
\item $\mathcal{C}+\mathcal{D}:=\mathcal{C}\cup\mathcal{D}$, if
$\mathcal{C}\cap\mathcal{D}=\varnothing$
\item $\mathcal{C}\times\mathcal{D}:=\{(c,d)|
c\in\mathcal{C},d\in\mathcal{D}\}$ and
$\mathcal{C}^m:=\prod_{i=1}^m\mathcal{C}$
\item $\textsc{Seq}(\mathcal{C})=
{\mathcal{E}}+\mathcal{C}+\mathcal{C}^2+\cdots$.
\end{itemize}
We have the following relations between the operations of
combinatorial classes and the operations of their generating
functions:
\begin{eqnarray}\label{E:symgf}
\mathcal{A}=\mathcal{C}+\mathcal{D} & \Rightarrow &
{\bf A}(z)={\bf C}(z)+{\bf D}(z)\\
\mathcal{A}=\mathcal{C}\times\mathcal{D}&\Rightarrow&
{\bf A}(z)={\bf C}(z)\cdot {\bf D}(z)\\
\mathcal{A}=
\textsc{Seq}(\mathcal{C})&\Rightarrow& {\bf A}(z)=(1-{\bf C}(z))^{-1},
\label{E:symgff}
\end{eqnarray}
where ${\bf A}(z)$, ${\bf C}(z)$, ${\bf D}(z)$ is the generating function
of $\mathcal{A}$, $\mathcal{C}$ and $\mathcal{D}$.

Given a $k$-noncrossing, $\tau$-canonical RNA structure $\delta$,
its ${\sf lv}_k^{\sf 5}$-shape, ${\sf lv}_k^{\sf 5}(\delta)$
\cite{Reidys:shape}, is obtained by first removing all isolated
vertices and second collapsing any stack into a single arc, see
Fig.\ref{F:I5shape}. By construction, ${\sf lv}_k^{\sf 5}$-shapes
do not preserve stack-lengths, interior loops and unpaired regions.
In the following, we shall refer to ${\sf lv}_k^{\sf 5}$-shape
simply as shape.
Let ${\mathcal T}_{k,\tau}$ denote the set of $k$-noncrossing,
$\tau$-canonical structures and ${\mathcal I}_k$ the set of
all $k$-noncrossing shapes and ${\mathcal I}_k(m)$ those having $m$
$1$-arcs, see Figure~\ref{F:I5shape}.
Each stem of a $k$-noncrossing, $\tau$-canonical RNA structure is
mapped into an arc in its corresponding shape and all hairpin-loops
are mapped into $1$-arcs. Therefore we have the surjective map,
\begin{equation}
\varphi:\ {\mathcal T}_{k,\tau}\rightarrow {\mathcal I}_k.
\end{equation}
%%%%%%%%%%%%%%%%%%%%%%%%%%%%%%%%%%%%%%%%%%%%%%%%%%%%%%%%%%%%%%%%%%%%%%%%
%%%%%%%%%%%%%%%%%%%%%%%%%%%%%%%%%%%%%%%%%%%%%%%%%%%%%%%%%%%%%%%%%%%%%%%%
\begin{figure}[ht]
\centerline{\includegraphics[width=0.6\textwidth]{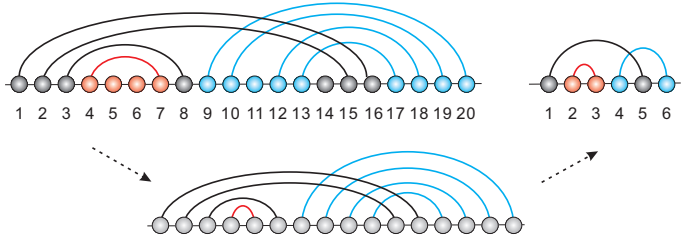}}
\caption{\small A $3$-noncrossing, $2$-canonical RNA structure
(top-left) is mapped into its shape (top-right)
in two steps. A stem (blue) is mapped into a single shape-arc (blue).
A hairpin-loop (red) is mapped into a shape-$1$-arc (red).}
\label{F:I5shape}
\end{figure}
%%%%%%%%%%%%%%%%%%%%%%%%%%%%%%%%%%%%%%%%%%%%%%%%%%%%%%%%%%%%%%%%%%%%%%%%
%%%%%%%%%%%%%%%%%%%%%%%%%%%%%%%%%%%%%%%%%%%%%%%%%%%%%%%%%%%%%%%%%%%%%%%%
Indeed, for a given shape $\gamma$ in ${\mathcal I}_k$, we can derive a
$k$-noncrossing, $\tau$-canonical structure having arc-length$\geq
2$, we can add arcs to each arc contained in the shape such that every
resulting stack has $\tau$ arcs and insert one isolated vertex in each
$1$-arc. Let
${\mathcal I}_k(s,m)$ and $i_k(s,m)$ denote the set and number of
the ${\sf lv}_k^{\sf 5}$-shapes of length $2s$ with $m$ $1$-arcs and
\begin{equation}
{\bf I}_k(x,y)=\sum_{s\geq0}\sum_{m=0}^{s} i_k(s,m)x^sy^m
\end{equation}
be the bivariate generating function. Furthermore, let $\mathcal{I}_k(m)$
denote the set of shapes $\gamma$ having $m$ $1$-arcs.
Let $k,s,m$ be natural numbers where $k\geq 2$, then the
generating function ${\bf I}_k(x,y)$ \cite{Reidys:shape} is given by
\begin{equation}\label{E:gfIk}
{\bf I}_k(x,y) =  \frac{1+x}{1+2x-xy}
{\bf F}_k\left(\frac{x(1+x)}{(1+2x-xy)^2}\right)
\end{equation}
%%%
%%%%%%%%%%%%%%%%%%%%%%%%%%%%%%%%%%%%%%%%%%%%%%%%%%%%%%%%%%%%%%%%%%%%%%%%
%%%
\begin{theorem}\label{T:hairpin}
Suppose $k,\tau\in \mathbb{N}, k\geq 2,\tau\geq1$. Then
\begin{equation}\label{E:bivagfh}
\begin{split}
{\bf H}_{k,\tau}(z,u_1)=&\frac{(1-z)(1-z^2+z^{2\tau})}
{(1-z)^2(1-z^2+z^{2\tau})+z^{2\tau}-z^{2\tau+1}u_1}\\
&{\bf F}_k\left(\frac{z^{2\tau}(1-z)^2(1-z^2+z^{2\tau})}
{\left((1-z)^2(1-z^2+z^{2\tau})
+z^{2\tau}-z^{2\tau+1}u_1\right)^2}\right),
\end{split}
\end{equation}
\begin{equation}
\begin{split}
{\bf I}_{k,\tau}(z,u_2)=&\frac{(1-z^2)(1-z)^2
-u_2z^{2\tau+2}+(2z^2-2z+1)z^{2\tau}} {(1-z)\left((1-z^2)(1-z)^2
-u_2z^{2\tau+2}+(2z^2-3z+2)z^{2\tau}\right)}\\
&{\bf F}_k\left(\frac{z^{2\tau}\left((1-z^2)(1-z)^2
-u_2z^{2\tau+2}+(2z^2-2z+1)z^{2\tau}\right)} {\left((1-z^2)(1-z)^2
-u_2z^{2\tau+2}+(2z^2-3z+2)z^{2\tau}\right)^2}\right),
\end{split}
\end{equation}
\begin{equation}
\begin{split}
{\bf
B}_{k,\tau}(z,u_3)=&\frac{(1-z^2)(1-z)-2u_3z^{2\tau+1}+(z+1)z^{2\tau}}
{(1-z)\left((1-z^2)(1-z)-2u_3z^{2\tau+1}+(z+2)z^{2\tau}\right)}\\
&{\bf F}_k\left(\frac{z^{2\tau}\left((1-z^2)(1-z)
-2u_3z^{2\tau+1}+(z+1)z^{2\tau}\right)}
{(1-z)\left((1-z^2)(1-z)-2u_3z^{2\tau+1}+(z+2)z^{2\tau}\right)^2}\right).
\end{split}
\end{equation}
\end{theorem}
%%%
%%%%%%%%%%%%%%%%%%%%%%%%%%%%%%%%%%%%%%%%%%%%%%%%%%%%%%%%%%%%%%%%%%%%%%%%
%%%
\begin{proof}
We prove the theorem via symbolic enumeration representing a $k$-noncrossing,
$\tau$-canonical structure as the inflation of a shape, $\gamma$.
Since a structure inflated from $\gamma\in \mathcal{I}_k(s,m)$ has
exactly $s$ stems, $(2s+1)$ (possibly empty) intervals of isolated
vertices and $m$ nonempty such intervals we rewrite the generating
functions as
\begin{eqnarray*}
{\bf H}_{k,\tau}(z,u_1)&=&\sum_{m\geq 0}\sum_{\gamma\in\,{\mathcal
I}_k(m)}{\bf T}_{\gamma}(z,u_1,1,1),\\
{\bf I}_{k,\tau}(z,u_2)&=&\sum_{m\geq 0}\sum_{\gamma\in\,{\mathcal
I}_k(m)}{\bf T}_{\gamma}(z,1,u_2,1),\\
{\bf B}_{k,\tau}(z,u_3)&=&\sum_{m\geq 0}\sum_{\gamma\in\,{\mathcal
I}_k(m)}{\bf T}_{\gamma}(z,1,1,u_3).
\end{eqnarray*}
where ${\bf T}_{\gamma}(z,u_1,u_2,u_3)$ is the generating function
of all $k$-noncrossing, $\tau$-canonical structures with shape
$\gamma$ and $u_i(i=1,2,3)$ are variables associated with the number of
hairpin-loops, interior-loops and bulges.
In order to compute the latter we consider the inflation
process:
we inflate $\gamma\in {\mathcal I}_k(m)$ having $s$ arcs,
where $s\geq m$, to a structure as follows:
\begin{itemize}
\item we inflate each arc of the shape to a stem of stacks of minimum size
$\tau$. Any isolated vertices inserted during this first inflation step
separate the added stacks,
\item we insert isolated vertices at the remaining $(2s+1)$ positions.
\end{itemize}
We inflate any shape-arc to a stack of size at least $\tau$ and
subsequently add additional stacks. The latter are called induced stacks
\index{induced stack} and have to be separated by means of inserting isolated
vertices, see Fig.~\ref{F:addstack}. Note that during this first inflation
step no intervals of isolated vertices, other than those necessary for
separating the nested stacks are inserted.
%%%%%%%%%%%%%%%%%%%%%%%%%%%%%%%%%%%%%%%%%%%%%%%%%%%%%%%%%%%%%%%%%%%%%%%%
%%%%%%%%%%%%%%%%%%%%%%%%%%%%%%%%%%%%%%%%%%%%%%%%%%%%%%%%%%%%%%%%%%%%%%%%
\begin{figure}[ht]
\centerline{\includegraphics[width=1\textwidth]{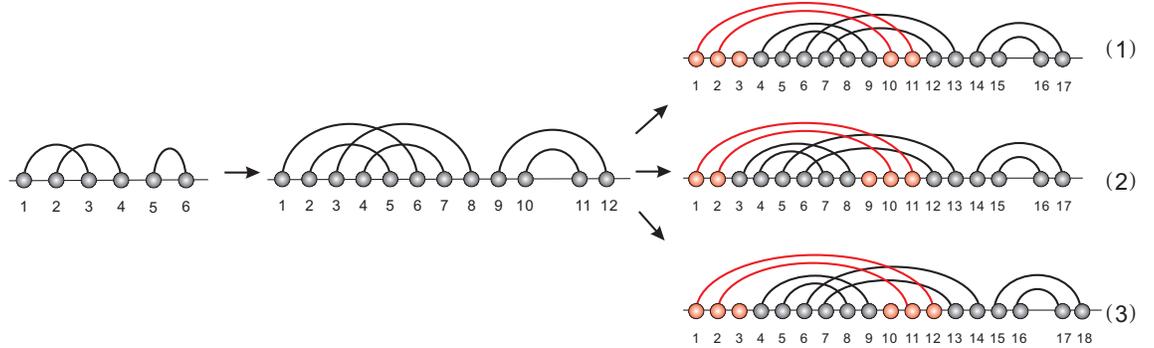}}
\caption{\small The first inflation step
a shape (left) is inflated to a $3$-noncrossing,
$2$-canonical structure. First, every arc in the shape is inflated to a
stack of size at least two (middle), and then the shape is inflated to a
new $3$-noncrossing, $2$-canonical structure (right) by adding one
stack of size two. There are three ways to insert the isolated vertices.}
\label{F:addstack}
\end{figure}
%%%%%%%%%%%%%%%%%%%%%%%%%%%%%%%%%%%%%%%%%%%%%%%%%%%%%%%%%%%%%%%%%%%%%%%%
%%%%%%%%%%%%%%%%%%%%%%%%%%%%%%%%%%%%%%%%%%%%%%%%%%%%%%%%%%%%%%%%%%%%%%%%
After the first inflation step we proceed inflating further by inserting
only additional isolated vertices at the remaining $(2s+1)$ positions
in which such insertions are possible. For each $1$-arc at least one
such isolated vertex is necessarily inserted, see Fig.~\ref{F:addvertex}.
%%%%%%%%%%%%%%%%%%%%%%%%%%%%%%%%%%%%%%%%%%%%%%%%%%%%%%%%%%%%%%%%%%%%%%%%
%%%%%%%%%%%%%%%%%%%%%%%%%%%%%%%%%%%%%%%%%%%%%%%%%%%%%%%%%%%%%%%%%%%%%%%%
\begin{figure}[ht]
\centerline{\includegraphics[width=1\textwidth]{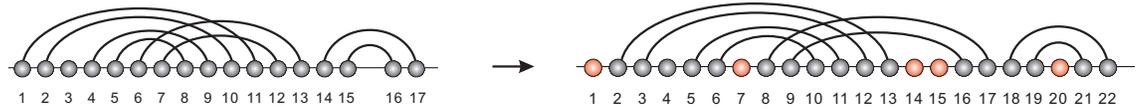}}
\caption{\small The second inflation step: the structure (left)
obtained in (1) in Fig.~\ref{F:addstack} is inflated to a new
$3$-noncrossing, $2$-canonical RNA structures (right) by adding isolated
vertices (red).} \label{F:addvertex}
\end{figure}
%%%%%%%%%%%%%%%%%%%%%%%%%%%%%%%%%%%%%%%%%%%%%%%%%%%%%%%%%%%%%%%%%%%%%%%%
%%%%%%%%%%%%%%%%%%%%%%%%%%%%%%%%%%%%%%%%%%%%%%%%%%%%%%%%%%%%%%%%%%%%%%%%
We proceed by expressing the above two inflations in terms of symbolic
enumeration. For this purpose we introduce the combinatorial classes
$\mathcal{M}$ (stems), $\mathcal{K}^{\tau}$ (stacks),
$\mathcal{N}^{\tau}$ (induced stacks), $\mathcal{L}$ (isolated vertices),
$\mathcal{R}$ (arcs) and $\mathcal{Z}$ (vertices), where ${\bf Z}(z)=z$
and ${\bf R}(z)=z^2$.
Let $\mu_1$, $\mu_2$ and $\mu_3$ be the labels for hairpin-loops,
interior-loops and
bulges, respectively. Then
\begin{eqnarray}
\mathcal{T}_{\gamma}&=&\left(\mathcal{M}\right)^s
\times\mathcal{L}^{2s+1-m}\times
\left([\mathcal{Z}\times\mathcal{L}]_{\mu_1}\right)^{m},\\
\mathcal{M}&=&\mathcal{K}^{\tau}\times
\textsc{Seq}\left(\mathcal{N}^{\tau}\right),\\
\mathcal{N}^{\tau}&=&\mathcal{K}^{\tau}\times
\left([\mathcal{Z}\times\mathcal{L}]_{\mu_3}
+[\mathcal{Z}\times\mathcal{L}]_{\mu_3}+[\left(\mathcal{Z}\times
\mathcal{L}\right)^2]_{\mu_2}\right),\\
\mathcal{K}^{\tau}&=&\mathcal{R}^{\tau}\times
\textsc{Seq}\left(\mathcal{R}\right),\\
\mathcal{L}&=& \textsc{Seq}\left(\mathcal{Z}\right).
\end{eqnarray}
and consequently, translating the above relations into generating functions
the generating function ${\bf T}_{\gamma}(z,u_1,u_2,u_3)$ is given by
\begin{equation}
\begin{split}
\left(\frac{\frac{z^{2\tau}}{1-z^2}}
{1-\frac{z^{2\tau}}{1-z^2}\left(2\frac{u_3\,z}{1-z}
+u_2\left(\frac{z}{1-z}\right)^2\right)}\right)^s
\left(\frac{1}{1-z}\right)^{2s+1-m}
\left(\frac{u_1\,z}{1-z}\right)^{m}\nonumber\\
= (1-z)^{-1}\left(\frac{z^{2\tau}}{(1-z^2)(1-z)^2-
(2u_3\,z\,(1-z)+u_2\,z^2)z^{2\tau}}\right)^s(u_1\,z)^m,
\end{split}
\end{equation}
where the indeterminants $u_i$ ($i=1,2,3$) correspond to the labels $\mu_i$,
i.e.~the occurrences of hairpin-loops, interior-loops and bulges.
Accordingly, for any two shapes $\gamma_1,\gamma_2\in
{\mathcal I}_k(m)$ having $s$ arcs, we have
\begin{equation}
{\bf T}_{\gamma_1}(z,u_1,u_2,u_3)={\bf T}_{\gamma_2}(z,u_1,u_2,u_3).
\end{equation}
We set
\begin{equation}
\eta(u_2,u_3)=\frac{z^{2\tau}}{(1-z^2)(1-z)^2-
(2u_3\,z\,(1-z)+u_2\,z^2)z^{2\tau}}.
\end{equation}
and accordingly derive
\begin{eqnarray*}
{\bf H}_{k,\tau}(z,u_1)&=&\sum_{m\geq 0}\sum_{\gamma\in\,{\mathcal
I}_k(m)}{\bf T}_{\gamma}(z,u_1,1,1)
=\sum_{s\geq 0}\sum_{m=0}^s i_k(s,m)
{\bf T}_{\gamma}(z,u_1,1,1),\\
{\bf I}_{k,\tau}(z,u_2)&=&\sum_{m\geq 0}\sum_{\gamma\in\,{\mathcal
I}_k(m)}{\bf T}_{\gamma}(z,1,u_2,1)=\sum_{s\geq 0}\sum_{m=0}^s
i_k(s,m){\bf
T}_{\gamma}(z,1,u_2,1),\\
{\bf B}_{k,\tau}(z,u_3)&=&\sum_{m\geq 0}\sum_{\gamma\in\,{\mathcal
I}_k(m)}{\bf T}_{\gamma}(z,1,1,u_3)=\sum_{s\geq 0}\sum_{m=0}^s
i_k(s,m){\bf T}_{\gamma}(z,1,1,u_3).
\end{eqnarray*}
It now remains to observe
\begin{equation*}
\sum_{s\geq0}\,\sum_{m=0}^{s}
\,i_k(s,m)\,x^s\,y^m=\frac{1+x}{1+2x-xy}{\bf F}_k
\left(\frac{x(1+x)}{(1+2x-xy)^2}\right).
\end{equation*}
and to subsequently substitute $x=\eta(1,1)$ and $y=u_1\,z$
for deriving ${\bf H}_{k,\tau}(z,u_1)$. Substituting $x=\eta(u_2,1)$
and $y=z$ in we obtain ${\bf I}_{k,\tau}(z,u_2)$ and finally
$x=\eta(1,u_3)$ and $y=z$ produce the expression for
${\bf B}_{k,\tau}(z,u_3)$, whence the theorem.
\end{proof}
%%%
%%%%%%%%%%%%%%%%%%%%%%%%%%%%%%%%%%%%%%%%%%%%%%%%%%%%%%%%%%%%%%%%%%%%%%%%
%%%
\section{The central limit theorem}
%%%
%%%%%%%%%%%%%%%%%%%%%%%%%%%%%%%%%%%%%%%%%%%%%%%%%%%%%%%%%%%%%%%%%%%%%%%%
%%%

For fixed $k$-noncrossing, $\tau$-canonical structure, $S$, let
$\mathbb{H}_{n,k,\tau}(S)$, $\mathbb{I}_{n,k,\tau}(S)$ and
$\mathbb{B}_{n,k,\tau}(S)$ denote the number of
hairpin-loops\index{hairpin-loop}, interior-loops\index{interior-loop}
and bulges\index{bulge} in $S$. Then we have the r.v.s
\begin{itemize}
\item $\mathbb{H}_{n,k,\tau}$, where
$\mathbb{P}\left(\mathbb{H}_{n,k,\tau}=t\right) =
\frac{h_{k,\tau}(n,t)}{{\sf T}_{k,\tau}(n)}$
\item $\mathbb{I}_{n,k,\tau}$, where
$\mathbb{P}\left(\mathbb{I}_{n,k,\tau}=t\right)=
\frac{i_{k,\tau}(n,t)}{{\sf T}_{k,\tau}(n)}$
\item $\mathbb{B}_{n,k,\tau}$, where
$\mathbb{P}\left(\mathbb{B}_{n,k,\tau}=t\right)
=\frac{b_{k,\tau}(n,t)}{{\sf T}_{k,\tau}(n)}.$
\end{itemize}
Here $h_{k,\tau}(n,t)$, $i_{k,\tau}(n,t)$ and $b_{k,\tau}(n,t)$
are the numbers of $k$-noncrossing, $\tau$-canonical structures
of length $n$ with $t$ hairpin-loops, interior-loops and bulges. The key
for computing the distributions of the above r.v.s are the bivariate
generating functions derived in Theorem~\ref{T:hairpin}:
\begin{eqnarray}
{\bf H}_{k,\tau}(z,u_1)&=&\sum_{n\geq 0}\sum_{t\geq 0}
h_{k,\tau}(n,t)z^n\,u_1^t,\\
{\bf I}_{k,\tau}(z,u_2)&=&\sum_{n\geq 0}\sum_{t\geq 0}
i_{k,\tau}(n,t)z^n\,u_2^t,\\
{\bf B}_{k,\tau}(z,u_3)&=&\sum_{n\geq 0}\sum_{t\geq
0}b_{k,\tau}(n,t)z^n\,u_3^t.
\end{eqnarray}
The following proposition is based on Theorem~\ref{T:uniform} and
facilitates the application of Theorem~\ref{T:normal}.
%%%
%%%%%%%%%%%%%%%%%%%%%%%%%%%%%%%%%%%%%%%%%%%%%%%%%%%%%%%%%%%%%%%%%%%%%%%%%%
%%%
\begin{proposition}\label{P:zcoeff}
Suppose  $2\le k\le 7$, $1\le \tau\le 10$. There exists a unique
dominant ${\bf H}_{k,\tau}(z,e^s)$-singularity, $\gamma_{k,\tau}(s)$,
such that for $|s|<\epsilon$, where $\epsilon>0$:\\
{\sf (1)} $\gamma_{k,\tau}(s)$ is analytic, \\
{\sf (2)} $\gamma_{k,\tau}(s)$ is the solution of minimal modulus of
\begin{equation}
\frac{z^{2\tau}(1-z)^2(1-z^2+z^{2\tau})}
{\left((1-z)^2(1-z^2+z^{2\tau})
+z^{2\tau}-z^{2\tau+1}e^s\right)^2}-\rho_k^2=0.
\end{equation}
and
\begin{equation}\label{E:coeffuni}
[z^n] {\bf H}_{k,\tau}(z,e^s)\sim C(s)\, n^{-((k-1)^2+\frac{k-1}{2})}
\left( \frac{1}{\gamma_{k,\tau}(s)} \right)^n,
\end{equation}
uniformly in $s$ in a neighborhood of $0$ and continuous $C(s)$.
\end{proposition}
%%%
%%%%%%%%%%%%%%%%%%%%%%%%%%%%%%%%%%%%%%%%%%%%%%%%%%%%%%%%%%%%%%%%%%%%%%%%%%
%%%
\begin{proof}
The first step is to establish the existence and uniqueness of the
dominant singularity $\gamma_{k,\tau}(s)$.\\
We denote
\begin{eqnarray}
\vartheta(z,s)&=&(1-z)^2(1-z^2+z^{2\tau})
+z^{2\tau}-z^{2\tau+1}e^s,\\
\psi_{\tau}(z,s)&=&z^{2\tau}
(1-z)^2(1-z^2+z^{2\tau})\vartheta(z,s)^{-2},\\
\omega_{\tau}(z,s)&=&(1-z)(1-z^2+z^{2\tau})\vartheta(z,s)^{-1},
\end{eqnarray}
and consider the equations
\begin{equation}
\forall \, 2\le i\le k;\qquad F_{i,\tau}(z,s)=\psi_{\tau}(z,s)-\rho_i^2,
\end{equation}
where $\rho_i=1/(2i-2)$.
Theorem~\ref{T:hairpin} and Table~\ref{Table:polyroot} imply that the
singularities of ${\bf H}_{k,\tau}(z,e^s)$ are
are contained in the set of roots of
\begin{equation}
F_{i,\tau}(z,s)=0\quad\text{and}\quad\vartheta(z,s)=0
\end{equation}
where $i\leq k$.
Let $r_{i,\tau}$ denote the solution of minimal modulus of
\begin{equation}
F_{i,\tau}(z,0)=\psi_{\tau}(z,0)-\rho_i^2=0.
\end{equation}
We next verify that, for sufficiently small $\epsilon_i>0$,
$|z-r_{i,\tau}|<\epsilon_i$, $|s|<\epsilon_i$, the following
assertions hold
\begin{itemize}
\item $\frac{\partial}{\partial z} F_{i,\tau}(r_{i,\tau},0)\neq 0$
\item $\frac{\partial}{\partial z} F_{i,\tau}(z,s)$ and
      $\frac{\partial}{\partial s}F_{i,\tau}(z,s)$ are
      continuous.
\end{itemize}
The analytic implicit function theorem, guarantees the existence of a unique
analytic function $\gamma_{i,\tau}(s)$ such that, for $|s|<\epsilon_i$,
\begin{equation}
F_{i,\tau}(\gamma_{i,\tau}(s),s)=0\quad\text{ and }\quad
\gamma_{i,\tau}(0)=r_{i,\tau}.
\end{equation}
Analogously, we obtain the unique analytic function $\delta(s)$ satisfying
$\vartheta(z,s)=0$ and where $\delta(0)$ is the minimal solution of
$\vartheta(z,0)=0$ for $|s|<\epsilon_{\delta}$, for some
$\epsilon_{\delta}>0$.
We next verify that the unique dominant singularity of
${\bf H}_{k,\tau}(z,1)$ is the minimal positive solution $r_{k,\tau}$ of
$F_{k,\tau}(z,0)=0$ and subsequently using an continuity argument.
Therefore, for sufficiently small $\epsilon$ where
$\epsilon<\epsilon_i$ and $\epsilon<\epsilon_{\delta}$,
$|s|<\epsilon$, the module of $\gamma_{i,\tau}(s)$, $i<k$
and $\delta(s)$ are all strictly larger than the modulus of
$\gamma_{k,\tau}(s)$. Consequently, $\gamma_{k,\tau}(s)$ is
the unique dominant singularity of ${\bf H}_{k,\tau}(z,e^s)$.\\
{\it Claim.} There exists some continuous $C(s)$ such that, uniformly in
$s$, for $s$ in a neighborhood of $0$
\begin{equation*}
[z^n] {\bf H}_{k,\tau}(z,e^s)\sim C(s)\, n^{-((k-1)^2+\frac{k-1}{2})}
\left( \frac{1}{\gamma_{k,\tau}(s)} \right)^n.
\end{equation*}
To prove the Claim, let $r$ be some positive real number such that
$r_{k,\tau}<r<\delta(0)$.
For sufficiently small $\epsilon>0$ and $|s|<\epsilon$,
$$
|\gamma_{k,\tau}(s)|\leq r\quad\text{and}\quad |\delta(s)|>r.
$$
Then $\psi_{\tau}(z,s)$ and $\omega_{\tau}(z,s)$ are all analytic in
$\mathcal{D}=\{(z,s)||z|\leq r,|s|<\epsilon\}$ and $\psi_{\tau}(0,s)=0$.
Since $\gamma_{k,\tau}(s)$ is the unique dominant singularity of
\begin{equation*}
{\bf H}_{k,\tau}(z,e^s)=\omega_{\tau}(z,s)
\,{\bf F}_k(\psi_{\tau}(z,s)),
\end{equation*}
satisfying
\begin{equation}
\psi_{\tau}(\gamma_{k,\tau}(s),s)=\rho_k^2\quad\text{and}\quad
|\gamma_{k,\tau}(s)|\leq r,
\end{equation}
for $|s|<\epsilon$. For sufficiently small $\epsilon>0$,
$\frac{\partial}{\partial z} F_{k,\tau}(z,s)$ is continuous and
$\frac{\partial}{\partial z}F_{k,\tau}(r_{k,\tau},0)\neq 0$.
Thus there exists some $\epsilon>0$, such that for $|s|<\epsilon$,
$\frac{\partial}{\partial z} F_{k,\tau}(\gamma_{k,\tau}(s),s)\neq 0$.
According to Theorem~\ref{T:uniform}, we therefore derive
\begin{equation}
[z^n] {\bf H}_{k,\tau}(z,e^s)\sim C(s)\, n^{-((k-1)^2+\frac{k-1}{2})}
\left( \frac{1}{\gamma_{k,\tau}(s)} \right)^n,
\end{equation}
uniformly in $s$ in a neighborhood of $0$ and continuous $C(s)$.
\end{proof}
%%%
%%%%%%%%%%%%%%%%%%%%%%%%%%%%%%%%%%%%%%%%%%%%%%%%%%%%%%%%%%%%%%%%%%%%%%%%%%
%%%

%%%
%%%%%%%%%%%%%%%%%%%%%%%%%%%%%%%%%%%%%%%%%%%%%%%%%%%%%%%%%%%%%%%%%%%%%%%%%%%%
%%%

After establishing the analogues of Proposition~\ref{P:zcoeff} for
$\mathbf{I}_{k,\tau}(z,u)$ and $\mathbf{B}_{k,\tau}(z,u)$, see the
Supplemental Materials, Theorem~\ref{T:normal} implies
the following central limit theorem for the distributions of
hairpin-loops, interior-loops\index{interior-loop} and
bulges\index{bulge} in $k$-noncrossing structures.
%%%
%%%%%%%%%%%%%%%%%%%%%%%%%%%%%%%%%%%%%%%%%%%%%%%%%%%%%%%%%%%%%%%%%%%%%%%%%
%%%
\begin{theorem}\label{T:gauss-hairpin}
Let $k,\tau\in \mathbb{N}$, $2\leq k\leq 7$, $1\leq\tau\leq 10$ and suppose the
random variable $\mathbb{X}$ denotes either $\mathbb{H}_{n,k,\tau}$,
$\mathbb{I}_{n,k,\tau}$ or $\mathbb{B}_{n,k,\tau}$.
Then there exists a pair
\begin{equation*}
(\mu_{k,\tau,\mathbb{X}},
\sigma_{k,\tau,\mathbb{X}}^{2})
\end{equation*}
such that the normalized random variable
$\mathbb{X}^*$ has asymptotically normal distribution
with parameter $(0,1)$, where $\mu_{k,\tau,\mathbb{X}}$ and
$\sigma_{k,\tau,\mathbb{X}}^2$ are given by
\begin{equation}\label{E:was-hairpin}
\mu_{k,\tau,\mathbb{X}}=
-\frac{\gamma_{k,\tau,\mathbb{X}}'(0)}{\gamma_{k,\tau,\mathbb{X}}(0)},
\qquad  \sigma_{k,\tau,\mathbb{X}}^2=
\left(\frac{\gamma_{k,\tau,\mathbb{X}}'(0)}{\gamma_{k,\tau,\mathbb{X}}(0)}
\right)^2-\frac{\gamma_{k,\tau,\mathbb{X}}''(0)}
{\gamma_{k,\tau,\mathbb{X}}(0)},
\end{equation}
where $\gamma_{k,\tau,\mathbb{X}}(s)$ represents
the unique dominant singularity of
 ${\bf H}_{k,\tau}(z,e^s)$, ${\bf I}_{k,\tau}(z,e^s)$, and
${\bf B}_{k,\tau}(z,e^s)$, respectively.
\end{theorem}

%%%
%%%%%%%%%%%%%%%%%%%%%%%%%%%%%%%%%%%%%%%%%%%%%%%%%%%%%%%%%%%%%%%%%%%%%%%%%
%%%
In Tables~\ref{Tab:data-hairpin}, \ref{Tab:data-interior} and
\ref{Tab:data-bulge} we present the values of the pairs
$(\mu_{k,\tau,\mathbb{X}},\sigma_{k,\tau,\mathbb{X}}^{2})$.

%%%%%%%%%%%%%%%%%%%%%%%%%%%%%%%%%%%%%%%%%%%%%%%%%%%%%%%%%%%
\begin{table}
\begin{center}
\begin{tabular}{|ccccccc|cc|}
\hline
&\multicolumn{2}{c}{$k=2$} &\multicolumn{2}{c}{$k=3$}        &
\multicolumn{2}{c|}{$k=4$} \\
\hline
& $\mu_{k,\tau}$ & $\sigma_{k,\tau}^2$ & $\mu_{k,\tau}$ & $\sigma_{k,\tau}^2$
& $\mu_{k,\tau}$ & $\sigma_{k,\tau}^2$ \\
\hline
$\tau=1$ & 0.105573 & 0.032260 & 0.012013 & 0.011202 & 0.003715 & 0.003641 \\
$\tau=2$ & 0.061281 & 0.018116 & 0.009845 & 0.008879 & 0.003734 & 0.003602 \\
$\tau=3$ & 0.043900 & 0.012752 & 0.007966 & 0.007060 & 0.003200 & 0.003060 \\
$\tau=4$ & 0.034477 & 0.009896 & 0.006680 & 0.005854 & 0.002757 & 0.002622 \\
\hline
&\multicolumn{2}{c}{$k=5$}  &\multicolumn{2}{c}{$k=6$}                        & \multicolumn{2}{c|}{$k=7$} \\
\hline
 & $\mu_{k,\tau}$ & $\sigma_{k,\tau}^2$ & $\mu_{k,\tau}$ & $\sigma_{k,\tau}^2$ & $\mu_{k,\tau}$ & $\sigma_{k,\tau}^2$ \\
\hline
$\tau=1$ & 0.001626 & 0.001612 & 0.000855 & 0.000852& 0.000505 & 0.000504 \\
$\tau=2$ & 0.001897 & 0.001864 & 0.001123 & 0.001111 & 0.000731 & 0.000726 \\
$\tau=3$ & 0.001693 & 0.001655 & 0.001035 & 0.001021 & 0.000692 & 0.000686 \\
$\tau=4$ & 0.001486 & 0.001448 & 0.000922 & 0.000907 & 0.000624 & 0.000618 \\
\hline
\end{tabular}
\centerline{}
\smallskip
\caption{\small {\sf Hairpin-loops:}
The central limit theorem for the numbers of hairpin-loops in $k$-noncrossing,
$\tau$-canonical structures. We list $\mu_{k,\tau}$ and
$\sigma^2_{k,\tau}$ derived from eq.~(\ref{E:was-hairpin}).}
\label{Tab:data-hairpin}
\end{center}
\end{table}
%%%
%%%%%%%%%%%%%%%%%%%%%%%%%%%%%%%%%%%%%%%%%%%%%%%%%%%%%%%%%%%%%%%%%%%%%%%%%%%%
%%%
%%%
%%%%%%%%%%%%%%%%%%%%%%%%%%%%%%%%%%%%%%%%%%%%%%%%%%%%%%%%%%%%%%%%%%%%%%%%%%%%
%%%

%%%
%%%%%%%%%%%%%%%%%%%%%%%%%%%%%%%%%%%%%%%%%%%%%%%%%%%%%%%%%%%%%%%%%%%%%%%%%%%%
%%%
\begin{table}
\begin{center}
\begin{tabular}{|ccccccc|cc|}
\hline
&\multicolumn{2}{c}{$k=2$} &\multicolumn{2}{c}{$k=3$}  &
\multicolumn{2}{c|}{$k=4$} \\
\hline
& $\mu_{k,\tau}$ & $\sigma_{k,\tau}^2$ & $\mu_{k,\tau}$ &
$\sigma_{k,\tau}^2$ & $\mu_{k,\tau}$ & $\sigma_{k,\tau}^2$ \\
\hline
$\tau=1$ & 0.015403 & 0.013916 & 0.001185 & 0.001176 & 0.000264 & 0.000264 \\
$\tau=2$ & 0.012959 & 0.011395 & 0.001823 & 0.001793 & 0.000603 & 0.000599 \\
$\tau=3$ & 0.011075& 0.009570 & 0.001878 & 0.001837 & 0.000693 & 0.000688 \\
$\tau=4$ & 0.009682 & 0.008261 & 0.001803 & 0.001755 & 0.000700 & 0.000693 \\
\hline
&\multicolumn{2}{c}{$k=5$}  &\multicolumn{2}{c}{$k=6$}                        & \multicolumn{2}{c|}{$k=7$} \\
\hline
 & $\mu_{k,\tau}$ & $\sigma_{k,\tau}^2$ & $\mu_{k,\tau}$ &
$\sigma_{k,\tau}^2$ & $\mu_{k,\tau}$ & $\sigma_{k,\tau}^2$ \\
\hline
$\tau=1$ & 0.000090 & 0.000090& 0.000039 & 0.000039& 0.000019 & 0.000019 \\
$\tau=2$ & 0.000275 & 0.000274  & 0.000149 & 0.000149 & 0.000090 & 0.000090 \\
$\tau=3$ & 0.000343 & 0.000341 & 0.000198 & 0.000198 & 0.000126 & 0.000126 \\
$\tau=4$ & 0.000359 & 0.000357 & 0.000214 & 0.000213 & 0.000140 & 0.000140 \\
\hline
\end{tabular}
\centerline{}
\smallskip
\caption{\small {\sf Interior-loops:}
The central limit theorem for the numbers of interior-loops in
$k$-noncrossing,$\tau$-canonical structures. We list $\mu_{k,\tau}$ and
$\sigma^2_{k,\tau}$ derived from eq.~(\ref{E:was-hairpin}).}
\label{Tab:data-interior}
\end{center}
\end{table}
%%%%%%%%%%%%%%%%%%%%%%%%%%%%%%%%%%%%%%%%%%%%%%%%%%%%%%%%%%%
%%%
%%%%%%%%%%%%%%%%%%%%%%%%%%%%%%%%%%%%%%%%%%%%%%%%%%%%%%%%%%%
\begin{table}
\begin{center}
\begin{tabular}{|ccccccc|cc|}
\hline
&\multicolumn{2}{c}{$k=2$} &\multicolumn{2}{c}{$k=3$}&
\multicolumn{2}{c|}{$k=4$} \\
\hline
& $\mu_{k,\tau}$ & $\sigma_{k,\tau}^2$ & $\mu_{k,\tau}$ &
$\sigma_{k,\tau}^2$ & $\mu_{k,\tau}$ & $\sigma_{k,\tau}^2$ \\
\hline
$\tau=1$ & 0.049845 & 0.042310 & 0.008982 & 0.008684 & 0.003094& 0.003058 \\
$\tau=2$ & 0.025088 & 0.021785 & 0.005789 & 0.005597 & 0.002457 & 0.002422 \\
$\tau=3$ & 0.015859 & 0.013979 & 0.003936 & 0.003814 & 0.001762 & 0.001737 \\
$\tau=4$ & 0.011197 & 0.009980 & 0.002878 &0.002795& 0.001318 & 0.001301  \\
\hline
&\multicolumn{2}{c}{$k=5$}  &\multicolumn{2}{c}{$k=6$} &
\multicolumn{2}{c|}{$k=7$} \\
\hline
& $\mu_{k,\tau}$ & $\sigma_{k,\tau}^2$ & $\mu_{k,\tau}$ & $\sigma_{k,\tau}^2$
& $\mu_{k,\tau}$ & $\sigma_{k,\tau}^2$ \\
\hline
$\tau=1$ & 0.001422 & 0.001414 & 0.000770 & 0.000767& 0.000463& 0.000462 \\
$\tau=2$ & 0.001326 & 0.001316 & 0.000817 & 0.000813 & 0.000547 & 0.000546 \\
$\tau=3$ & 0.000991 & 0.000984 & 0.000632 & 0.000629 & 0.000436 & 0.000435 \\
$\tau=4$ & 0.000755 & 0.000750 & 0.000489 & 0.000486 & 0.000342 & 0.000341\\
\hline
\end{tabular}
\centerline{}
\smallskip
\caption{\small {\sf Bulges:}
The central limit theorems for the numbers of bulges in $k$-noncrossing,
$\tau$-canonical structures. We list $\mu_{k,\tau}$ and
$\sigma^2_{k,\tau}$ derived from eq.~(\ref{E:was-hairpin}).}
\label{Tab:data-bulge}
\end{center}
\end{table}
%%%%%%%%%%%%%%%%%%%%%%%%%%%%%%%%%%%%%%%%%%%%%%%%%%%%%%%%%%%
%%%%%%%%%%%%%%%%%%%%%%%%%%%%%%%%%%%%%%%%%%%%%%%%%%%%%%%%%%%

\section{Proof of Theorem~\ref{T:uniform}}\label{S:proof}

{\textbf{Proof of Theorem~\ref{T:uniform}.}}
We consider the composite function ${\bf F}_k(\psi(z,s))$. In view of
$[z^n]f(z,s)=\gamma^n [z^n]f(\frac{z}{\gamma},s)$ it suffices to analyze
the function ${\bf F}_k(\psi(\gamma(s)z,s))$ and to subsequently rescale
in order to obtain the correct exponential factor. For this purpose we set
\begin{equation*}
\widetilde{\psi}(z,s)=\psi(\gamma(s)z,s),
\end{equation*}
where $\psi(z,s)$ is analytic in a domain
$\mathcal{D}=\{(z,s)||z|\leq r,|s|<\epsilon\}$. Consequently
$\widetilde{\psi}(z,s)$ is analytic in $|z|<\widetilde{r}$ and
$|s|<\widetilde{\epsilon}$, for some
$1<\widetilde{r},\,0<\widetilde{\epsilon}<\epsilon$,
since it's a composition
of two analytic functions in $\mathcal{D}$. Taking its Taylor
expansion at $z=1$,
\begin{equation}\label{E:taylor}
\widetilde{\psi}(z,s)=\sum_{n\geq0}\widetilde{\psi}_n(s)(1-z)^n,
\end{equation}
where $\widetilde{\psi}_n(s)$ is analytic in $|s|<\widetilde{\epsilon}$.
The singular expansion of ${\bf F}_k(z)$, $2\leq k\leq 7$,
for $z\rightarrow \rho_k^2$, follows from the ODEs, see eq.~(\ref{E:ODE}),
and is given by
\begin{equation}
{\bf F}_k(z)=
\begin{cases}
P_k(z-\rho_k^2)+c_k'(z-\rho_k^2)^{((k-1)^2+(k-1)/2)-1}
\log^{}(z-\rho_k^2)\left(1+o(1)\right)  \\
P_k(z-\rho_k^2)+c_k'(z-\rho_k^2)^{((k-1)^2+(k-1)/2)-1}
\left(1+o(1)\right)
\end{cases}
\end{equation}
depending on whether $k$ is odd or even and where $P_k(z)$ are polynomials
of degree $\le (k-1)^2+(k-1)/2-1$, $c_k'$ is some constant, and
$\rho_k=1/2(k-1)$.
By assumption, $\gamma(s)$ is the unique analytic solution of
$\psi(\gamma(s),s)=\rho_k^2$ and by construction
${\bf F}_k(\psi(\gamma(s)z,s))={\bf F}_k(\widetilde{\psi}(z,s))$.
In view of eq.~(\ref{E:taylor}), we have for $z\rightarrow 1$ the
expansion
\begin{equation}\label{E:uinner}
\widetilde{\psi}(z,s)-\rho_k^2
=\sum_{n\geq1}\widetilde{\psi}_n(s)(1-z)^n= \widetilde{\psi}_1(s)(1-z)(1+o(1)),
\end{equation}
that is uniform in $s$ since $\widetilde{\psi}_n(s)$ is analytic for
$|s|<\widetilde{\epsilon}$ and
$\widetilde{\psi}_0(s)=\psi(\gamma(s),s)=\rho_k^2$.
As for the singular expansion of ${\bf F}_k(\widetilde{\psi}(z,s))$
we derive, substituting the eq.~(\ref{E:uinner})
into the singular expansion of ${\bf F}_k(z)$, for $z\rightarrow 1$,
\begin{equation}
\begin{cases}
\widetilde{P}_k(z,s)+c_k(s)(1-z)^{((k-1)^2+(k-1)/2)-1}
\log^{}(1-z)\left(1+o(1)\right) & \text{\rm for
$k$ odd} \\
\widetilde{P}_k(z,s)+c_k(s)(1-z)^{((k-1)^2+(k-1)/2)-1}
\left(1+o(1)\right) & \text{\rm for $k$ even}
\end{cases}
\end{equation}
where $\widetilde{P}_k(z,s)=P_k(\widetilde{\psi}(z,s)-\rho_k^2)$ and
$c_k(s)=c_k'\widetilde{\psi}_1(s)^{((k-1)^2+(k-1)/2)-1}$ and
\begin{equation*}
\widetilde{\psi}_1(s)=\partial_z\widetilde{\psi}(z,s)|_{z=1}
=\gamma(s)\partial_z\psi(\gamma(s),s)\neq 0 \quad \text{\rm for }
\vert s\vert<\epsilon.
\end{equation*}
Furthermore $\widetilde{P}_k(z,s)$ is analytic at
$|z|\leq 1$, whence $[z^n]\widetilde{P}_k(z,s)$ is exponentially small
compared to $1$.
Therefore we arrive at
\begin{equation}\label{E:coeffbiva}
[z^n]{\bf F}_k(\widetilde{\psi}(z,s))\sim
\begin{cases}
[z^n]c_k(s)(1-z)^{((k-1)^2+(k-1)/2)-1}
\log^{}(1-z)\left(1+o(1)\right)  \\
[z^n]c_k(s)(1-z)^{((k-1)^2+(k-1)/2)-1} \left(1+o(1)\right)
\end{cases}
\end{equation}
depending on $k$ being odd or even and uniformly in $|s|<\widetilde{\epsilon}$.
We observe that $c_k(s)$ is analytic in $|s|<\widetilde{\epsilon}$.
Note that a dependency in the parameter $s$ is only given in the
coefficients $c_k(s)$, that are analytic in $s$.
Standard transfer theorems \cite{Flajolet:07a} imply that
\begin{equation}
[z^n]{\bf F}_k(\widetilde{\psi}(z,s))\sim
A(s)\,n^{-((k-1)^2+(k-1)/2)}\quad \text{\rm for some
$A(s)\in\mathbb{C}$},
\end{equation}
uniformly in $s$ contained in a small neighborhood of $0$.
Finally, as mention in the beginning of the proof, we use the scaling property
of Taylor expansions in order to derive
\begin{equation}
[z^n]{\bf F}_k(\psi(z,s))=\left(\gamma(s)\right)^{-n}[z^n]{\bf
F}_k(\widetilde{\psi}(z,s))
\end{equation}
and the proof of the Theorem is complete.

%%%%%%%%%%%%%%%%%%%%%%%%%%%%%%%%%%%%%%%%%%%%%%%%%%%%%%%%%%%%%%%%%%%%%%%%%%
%%%
{\bf Acknowledgments.}
%%%
%%%%%%%%%%%%%%%%%%%%%%%%%%%%%%%%%%%%%%%%%%%%%%%%%%%%%%%%%%%%%%%%%%%%%%%%%%
%%%
This work was supported by the 973 Project, the PCSIRT Project of
the Ministry of Education, the Ministry of Science and Technology,
and the National Science Foundation of China.

\bibliographystyle{plain}

\end{document}